\documentclass[aip,amsfonts,amsmath,prd,preprint,nofootinbib]{revtex4}

\usepackage[utf8]{inputenc}
\usepackage{amsfonts}
\usepackage{physics}
\usepackage{mathtools}
\usepackage{mathrsfs}
\usepackage{amssymb}
\usepackage{appendix}
\usepackage{graphicx}
\usepackage{amsmath}
\usepackage{color}
\usepackage{feynmf}
\usepackage{simpler-wick} 
\usepackage[legalpaper, margin=1in]{geometry}
\usepackage{amsthm}
\usepackage[makeroom]{cancel}
\usepackage{hyperref}

\newtheorem{theorem}{Theorem}[section]

\newtheorem{lemma}[theorem]{Lemma}
\newtheorem{proposition}[theorem]{Proposition}

\begin{document}

\title{Sums of products of Bessel functions and order derivatives of Bessel functions}

\author{Yilin Chen}
\email{Yilin.Chen@tufts.edu}
\address{Institute of Cosmology, Department of Physics and Astronomy, \\
Tufts University, Medford, Massachusetts 02155, USA}

\begin{abstract}
    In this paper, sums represented in \eqref{sum} are studied. The expressions are derived in terms of Bessel functions of the first and second kinds and their integrals. Further, we point out the integrals can be written as a Meijer G function.
\end{abstract}
\maketitle

\section{Introduction}
In \cite{chen}, we find a closed-form for a summation of series involving Bessel functions and the order derivatives of Bessel functions. We also apply this expression in the study of entanglement entropy in two dimensional bosonic free field. However, there is a problem left unsolved: How we can find a concise way to evaluate
\begin{align}
    (2\partial_{\omega})^i(2\partial_{\bar{\omega}})^j \left[\left.\frac{\partial}{\partial n}\Tilde{G}_n^{\left(1+\frac{i+j+l}{2}\right)}(\mathbf{r},\mathbf{r'})\right|_{n=1}-\Tilde{G}_1^{\left(1+\frac{i+j+l}{2}\right)}(\mathbf{r},\mathbf{r'})\right]_{\mathbf{r}=\mathbf{r'}}
\end{align}
when $i\neq j$, just like the case in \cite{chen} for $i=j$? The major problem is that the derivative $2\partial_{\omega}$ or $2\partial_{\bar{\omega}}$ would break the ‘paired-up’ form of the summation $\sum_{n=0}^{\infty}nJ_n(x)\partial J_{n}(x)/\partial n$, i.e.
\begin{align}
    2\partial_{\omega}\left(e^{in\theta}J_{n}(\lambda r)\right)=\lambda e^{i(n-1)\theta}J_{n-1}(\lambda r),\ \ 2\partial_{\omega}\left(e^{-in\theta}J_{n}(\lambda r)\right)=-\lambda e^{-i(n+1)\theta}J_{n+1}(\lambda r),\nonumber\\
    2\partial_{\bar{\omega}}\left(e^{in\theta}J_{n}(\lambda r)\right)=-\lambda e^{i(n+1)\theta}J_{n+1}(\lambda r),\ \ 2\partial_{\bar{\omega}}\left(e^{-in\theta}J_{n}(\lambda r)\right)=\lambda e^{-i(n-1)\theta}J_{n-1}(\lambda r),\nonumber
\end{align}
so that
\begin{align}
    (2\partial_{\omega})^{\mu+\nu}(2\partial_{\bar{\omega}})^{\nu} \left[\left.\frac{\partial}{\partial n}\Tilde{G}_n^{\left(1+\nu+\frac{\mu+l}{2}\right)}(\mathbf{r},\mathbf{r'})\right|_{n=1}-\Tilde{G}_1^{\left((1+\nu+\frac{\mu+l}{2}\right)}(\mathbf{r},\mathbf{r'})\right]_{\mathbf{r}=\mathbf{r'}}\nonumber\\
    =\frac{1}{\pi}\int_0^{\infty}\frac{(-\lambda)^{2\nu+1}}{\left(\lambda^2+m^{'2}\right)^{1+\nu+\frac{\mu+l}{2}}}e^{-i\mu\theta}P_{\mu}(\lambda r)d\lambda,
\end{align}
where
\begin{align}\label{sum}
    P_{\mu}(x)\equiv\sum_{n=1}^{\infty}n\left[\left(\hat{J}_{n-\mu}(x)J_{n}(x)+J_{n-\mu}(x)\hat{J}_{n}(x)\right)+(-1)^{\mu}\left(\hat{J}_{n+\mu}(x)J_{n}(x)+J_{n+\mu}(x)\hat{J}_{n}(x)\right)\right],
\end{align}
where hats indicate the order derivatives. Now, we can point out that the crucial point to solve this problem is how we can get a concise form of $P_{\mu}(x)$.

In this paper, we study the summation $P_{\mu}(x)$ and obtain a closed-form expression of it. The main work is done in Section II but with an undetermined integral constant. Later, in the next section, we analyze the asymptotic behavior of $P_{\mu}(x)$ and determine the integral constant. In the end, we evaluate the integrals in the expression and represent $P_{\mu}(x)$ in terms of Bessel functions and Meijer G functions.
\section{Derivation of $P_{\mu}(x)$}
\begin{lemma}
For positive $x$ and $\mu,\nu\in\mathbb{Z}$, 
\begin{align}\label{lammam1}
    \sum_{n=0}^{\infty}\varepsilon_n J_{\nu+\mu+n}(x)J_{\nu+n}(x)&=\int_0^x\frac{2\nu+\mu}{t}J_{\nu+\mu}(t)J_{\nu}(t)dt\nonumber\\
    &=\frac{2}{\pi}\frac{\sin\frac{\pi}{2}\mu}{\mu}-\int_x^{\infty}\frac{2\nu+\mu}{t}J_{\nu+\mu}(t)J_{\nu}(t)dt,
\end{align}
where $\varepsilon_0=1$ and $\varepsilon_n=2$ otherwise.
\end{lemma}
\begin{proof}
Firstly, using
\begin{align}
    J'_p(x)=\frac{1}{2}\left(J_{p-1}(x)-J_{p+1}(x)\right),
\end{align}
we have
\begin{align}
    \{J_{\nu+\mu+n}(x)J_{\nu+n}(x)\}'=&\frac{1}{2}\left(J_{\nu+\mu+n}(x)J_{\nu+n-1}(x)-J_{\nu+\mu+n+1}(x)J_{\nu+n}(x)\right)\nonumber\\
    &+\frac{1}{2}\left(J_{\nu+\mu+n-1}(x)J_{\nu+n}(x)-J_{\nu+\mu+n}(x)J_{\nu+n+1}(x)\right).
\end{align}
Summing both sides over $n=0$ to $n=N$ gives, via telescoping
\begin{align}\label{lemma1}
    \sum_{n=0}^N\{J_{\nu+\mu+n}(x)J_{\nu+n}(x)\}'=&\frac{1}{2}\left(J_{\nu+\mu}(x)J_{\nu-1}(x)+J_{\nu+\mu-1}(x)J_{\nu}(x)\right)\nonumber\\
    &-\frac{1}{2}\left(J_{\nu+\mu+N+1}(x)J_{\nu+N}(x)+J_{\nu+\mu+N}(x)J_{\nu+N+1}(x)\right).
\end{align}
Before setting $N\to\infty$, we use
\begin{align}
    J_{p-1}(x)=\frac{p}{x}J_{p}(x)+J'_{p}(x)
\end{align}
to obtain from \eqref{lemma1},
\begin{align}\label{lemma2}
    \sum_{n=0}^N\{J_{\nu+\mu+n}(x)J_{\nu+n}(x)\}'=&\frac{\nu+\mu/2}{x}J_{\nu}(x)J_{\nu+\mu}(x)+\frac{1}{2}\left\{J_{\nu+\mu}(x)J_{\nu}(x)\right\}'\nonumber\\
    &-\frac{\nu+N+1+\mu/2}{x}J_{\nu+\mu+N+1}(x)J_{\nu+N+1}(x)\nonumber\\
    &-\frac{1}{2}\left\{J_{\nu+\mu+N+1}(x)J_{\nu+N+1}(x)\right\}'.
\end{align}
When $N\to\infty$, for finite $x$, it is safe to apply the asymptotic approximation for small argument $0<x<<\sqrt{\alpha+1}$, which gives
\begin{align}
    J_{\alpha}(x)\sim\frac{1}{\Gamma(\alpha+1)}\left(\frac{x}{2}\right)^{\alpha}.
\end{align}
Hence, the terms which contain orders of $N$ in \eqref{lemma2} go to zero when $N$ approaches to infinity and \eqref{lemma2} becomes
\begin{align}\label{lamma3}
    \frac{1}{2}\sum_{n=0}^N\varepsilon_n\{J_{\nu+\mu+n}(x)J_{\nu+n}(x)\}'=&\frac{\nu+\mu/2}{x}J_{\nu}(x)J_{\nu+\mu}(x).
\end{align}
Then integrating \eqref{lamma3} over $0$ to $x$ yields the first equation in \eqref{lammam1}. Recalling (see \cite{GR1}, 6.574)
\begin{align}
    \int_0^{\infty}\frac{dt}{t} J_{p}(t)J_{q}(t)=\frac{2}{\pi}\frac{\sin\left(\frac{\pi}{2}(p-q)\right)}{p^2-q^2},
\end{align}
the integral over $0$ to $\infty$ is convergent. By subtracting this definite integral, we have the second line in \eqref{lammam1}.
\end{proof}

\begin{lemma}
For integer $\mu$,
\begin{align}\label{lammam2}
    \sum_{n=1}^{\infty}\left(J_{n-\mu}(x)J_{n}(x)+(-1)^{\mu}J_{n+\mu}(x)J_{n}(x)\right)+J_{-\mu}(x)J_0(x)=\delta_{\mu0}.
\end{align}
\end{lemma}
\begin{proof}
Using $J_{-n}(x)=(-1)^nJ_n(x)$ for integer $n$ and the Graf's additional theorem yields
\begin{align}
    \sum_{n=1}^{\infty}\left(J_{n-\mu}(x)J_{n}(x)+(-1)^{\mu}J_{n+\mu}(x)J_{n}(x)\right)+J_{-\mu}(x)J_0(x)=\sum_{n=-\infty}^{\infty}J_{n-\nu}(x)J_{n}(x)=J_{-\mu}(0).
\end{align}
\end{proof}

\begin{proposition}
For non-zero and integer $\mu$
\begin{align}\label{propm}
    P_{\mu}(x)=(-1)^{\mu}\left(-f_{\mu}(x)+\mu^2\int_x^{\infty}\frac{f_{\mu}(t)}{t}dt+(1-\mu^2)x\int_x^{\infty}\frac{f_{\mu}(t)}{t^2}dt+Cx\right),
\end{align}
where $C$ is an integral constant.
\end{proposition}
\begin{proof}
Let us start with \eqref{lammam1}. By taking the order derivative on both sides, we have
\begin{align}
    &2\sum_{n=1}^{\infty}\left(\hat{J}_{\mu+\nu+n}(x)J_{n+\nu}(x)\!+\!J_{\mu+\nu+n}(x)\hat{J}_{n+\nu}(x)\right)+\left(\hat{J}_{\mu+\nu}(x)J_{\nu}(x)\!+\!J_{\mu+\nu}(x)\hat{J}_{\nu}(x)\right)\nonumber\\
    &=-\int_x^{\infty}\left(2J_{\nu+\mu}(t)J_{\nu}(t)\!+\!(2\nu+\mu)\left(\hat{J}_{\mu+\nu}(t)J_{\nu}(t)\!+\!J_{\mu+\nu}(x)\hat{J}_{\nu}(t)\right)\right)\frac{dt}{t}.
\end{align}
Then, after simple manipulation and summing over $\nu=0,1,2,...$, we abtain
\begin{align}\label{prop1}
    &\sum_{\nu=0}^{\infty}\sum_{n=1}^{\infty}\left(\hat{J}_{\mu+\nu+n}(x)J_{n+\nu}(x)\!+\!J_{\mu+\nu+n}(x)\hat{J}_{n+\nu}(x)\right)\nonumber\\
    &=\sum_{n=1}^{\infty}n\left(\hat{J}_{\mu+n}(x)J_{n}(x)\!+\!J_{\mu+n}(x)\hat{J}_{n}(x)\right)\nonumber\\
    &=-\int_x^{\infty}\left(\sum_{\nu=0}^{\infty}J_{\nu+\mu}(t)J_{\nu}(t)\!+\!\sum_{\nu=0}^{\infty}(\nu+\frac{\mu}{2})\left(\hat{J}_{\mu+\nu}(t)J_{\nu}(t)\!+\!J_{\mu+\nu}(x)\hat{J}_{\nu}(t)\right)\right)\frac{dt}{t}\nonumber\\
    &\ \ \ \ -\frac{1}{2}\sum_{\nu=0}^{\infty}\left(\hat{J}_{\mu+\nu}(x)J_{\nu}(x)\!+\!J_{\mu+\nu}(x)\hat{J}_{\nu}(x)\right).
\end{align}
Now we can construct $P_{\mu}$ from the LHS of \eqref{prop1}. Letting $\mu\to-\mu$ in \eqref{prop1} and then adding $(-1)^{\mu}\times$\eqref{prop1} with the original $\mu$ yields
\begin{align}\label{prop2}
    &P_{\mu}(x)\equiv\sum_{n=1}^{\infty}\left\{n\left[\left(\hat{J}_{n-\mu}(x)J_{n}(x)+J_{n-\mu}(x)\hat{J}_{n}(x)\right)+(-1)^{\mu}\left(\hat{J}_{n+\mu}(x)J_{n}(x)+J_{n+\mu}(x)\hat{J}_{n}(x)\right)\right]\right\}\nonumber\\
    &=-\int_x^{\infty}\left(\sum_{\nu=0}^{\infty}\left(J_{\nu-\mu}(t)J_{\nu}(t)+(-1)^{\mu}J_{\nu+\mu}(t)J_{\nu}(t)\right)\right)\frac{dt}{t}\nonumber\\
    &\ \ \ \ -\int_x^{\infty}\left(\sum_{\nu=0}^{\infty}\nu\left(\hat{J}_{-\mu+\nu}(t)J_{\nu}(t)\!+\!J_{-\mu+\nu}(x)\hat{J}_{\nu}(t)\right)+(-1)^{\mu}\left(\hat{J}_{\mu+\nu}(t)J_{\nu}(t)\!+\!J_{\mu+\nu}(x)\hat{J}_{\nu}(t)\right)\right)\frac{dt}{t}\nonumber\\
    &\ \ \ \ -\int_x^{\infty}\left(\sum_{\nu=0}^{\infty}\frac{\mu}{2}\left(-\left(\hat{J}_{-\mu+\nu}(t)J_{\nu}(t)\!+\!J_{-\mu+\nu}(x)\hat{J}_{\nu}(t)\right)+(-1)^{\mu}\left(\hat{J}_{\mu+\nu}(t)J_{\nu}(t)\!+\!J_{\mu+\nu}(x)\hat{J}_{\nu}(t)\right)\right)\right)\frac{dt}{t}\nonumber\\
    &\ \ \ \ -\frac{1}{2}\sum_{\nu=0}^{\infty}\left(\left(\hat{J}_{-\mu+\nu}(x)J_{\nu}(x)\!+\!J_{-\mu+\nu}(x)\hat{J}_{\nu}(x)\right)+(-1)^{\mu}\left(\hat{J}_{\mu+\nu}(x)J_{\nu}(x)\!+\!J_{\mu+\nu}(x)\hat{J}_{\nu}(x)\right)\right).
\end{align}
We can deal with the RHS of \eqref{prop2} term by term. For the first term, we apply \eqref{lammam2}, which gives
\begin{align}
    &-\int_x^{\infty}\left(\sum_{\nu=0}^{\infty}\left(J_{\nu-\mu}(t)J_{\nu}(t)+(-1)^{\mu}J_{\nu+\mu}(t)J_{\nu}(t)\right)\right)\frac{dt}{t}\nonumber\\
    =&-\int_x^{\infty}\left(\sum_{\nu=1}^{\infty}\left(J_{\nu-\mu}(t)J_{\nu}(t)+(-1)^{\mu}J_{\nu+\mu}(t)J_{\nu}(t)\right)+J_{-\mu}(t)J_{0}(t)+(-1)^{\mu}J_{\mu}(t)J_{0}(t)\right)\frac{dt}{t}\nonumber\\
    =&-(-1)^{\mu}\int_x^{\infty} J_{\mu}(t)J_{0}(t)\frac{dt}{t}.
\end{align}
The second term is rather simple:
\begin{align}
     -&\int_x^{\infty}\left(\sum_{\nu=0}^{\infty}\nu\left(\hat{J}_{-\mu+\nu}(t)J_{\nu}(t)\!+\!J_{-\mu+\nu}(x)\hat{J}_{\nu}(t)\right)+(-1)^{\mu}\left(\hat{J}_{\mu+\nu}(t)J_{\nu}(t)\!+\!J_{\mu+\nu}(x)\hat{J}_{\nu}(t)\right)\right)\frac{dt}{t}\nonumber\\
     &=-\int_x^{\infty}\frac{P_{\mu}(t)}{t}dt.
\end{align}
For the integrand of the third term, using \eqref{lammam1}, we have
\begin{align}
    &\sum_{\nu=0}^{\infty}\left(-\left(\hat{J}_{-\mu+\nu}(t)J_{\nu}(t)\!+\!J_{-\mu+\nu}(x)\hat{J}_{\nu}(t)\right)+(-1)^{\mu}\left(\hat{J}_{\mu+\nu}(t)J_{\nu}(t)\!+\!J_{\mu+\nu}(x)\hat{J}_{\nu}(t)\right)\right)\nonumber\\
    =&\int_x^{\infty}\left(J_{-\mu}(t)J_{0}(t)\!-\!\frac{\mu}{2}\left(\hat{J}_{-\mu}(t)J_{0}(t)\!+\!J_{-\mu}(x)\hat{J}_{0}(t)\right)\right)\frac{dt}{t}-\frac{1}{2}\left(\hat{J}_{-\mu}(x)J_{0}(x)\!+\!J_{-\mu}(x)\hat{J}_{0}(x)\right)\nonumber\\
    &\ +(-1)^{\mu}\left(-\int_x^{\infty}\left(J_{\mu}(t)J_{0}(t)\!+\!\frac{\mu}{2}\left(\hat{J}_{\mu}(t)J_{0}(t)\!+\!J_{\mu}(x)\hat{J}_{0}(t)\right)\right)\frac{dt}{t}+\frac{1}{2}\left(\hat{J}_{\mu}(x)J_{0}(x)\!+\!J_{\mu}(x)\hat{J}_{0}(x)\right)\right).
\end{align}
From the representations of $\hat{J}_n(x)$ with integer $n$ (see \cite{GR}, 8.486(1)), we can obtain such a relation
\begin{align}\label{prop3}
    \hat{J}_{-\mu}(x)+(-1)^{\mu}\hat{J}_{\mu}(x)=(-1)^{\mu}\pi Y_{\mu}(x).
\end{align}
Therefore, by applying \eqref{prop3}, $\hat{J}_0(x)=\pi Y_0(x)/2$ and $J_{-n}(x)=(-1)^nJ_{n}(x)$ for integer $n$, we have
\begin{align}
    &\sum_{\nu=0}^{\infty}\left(\left(-\hat{J}_{-\mu+\nu}(t)J_{\nu}(t)\!+\!J_{-\mu+\nu}(x)\hat{J}_{\nu}(t)\right)+(-1)^{\mu}\left(\hat{J}_{\mu+\nu}(t)J_{\nu}(t)\!+\!J_{\mu+\nu}(x)\hat{J}_{\nu}(t)\right)\right)\nonumber\\
    =&-(-1)^{\mu}\frac{\pi\mu}{2}\int_x^{\infty}\left( Y_{\mu}(t)J_{0}(t)\!+\!J_{\mu}(t)Y_{0}(t)\right)\frac{dt}{t}-\frac{1}{2}\left(\hat{J}_{-\mu}(x)-(-1)^{\mu}\hat{J}_{\mu}(x)\right)J_{0}(x).
\end{align}
For the last term, we can follow the same idea in handling the third term. Then, the last term turns
\begin{align}
    &\sum_{\nu=0}^{\infty}\left(\left(\hat{J}_{-\mu+\nu}(t)J_{\nu}(t)\!+\!J_{-\mu+\nu}(x)\hat{J}_{\nu}(t)\right)+(-1)^{\mu}\left(\hat{J}_{\mu+\nu}(t)J_{\nu}(t)\!+\!J_{\mu+\nu}(x)\hat{J}_{\nu}(t)\right)\right)\nonumber\\
    =&-\int_x^{\infty}\left(J_{-\mu}(t)J_{0}(t)\!-\!\frac{\mu}{2}\left(\hat{J}_{-\mu}(t)J_{0}(t)\!+\!J_{-\mu}(x)\hat{J}_{0}(t)\right)\right)\frac{dt}{t}+\frac{1}{2}\left(\hat{J}_{-\mu}(x)J_{0}(x)\!+\!J_{-\mu}(x)\hat{J}_{0}(x)\right)\nonumber\\
    &\ +(-1)^{\mu}\left(-\int_x^{\infty}\left(J_{\mu}(t)J_{0}(t)\!+\!\frac{\mu}{2}\left(\hat{J}_{\mu}(t)J_{0}(t)\!+\!J_{\mu}(x)\hat{J}_{0}(t)\right)\right)\frac{dt}{t}+\frac{1}{2}\left(\hat{J}_{\mu}(t)J_{0}(x)\!+\!J_{\mu}(x)\hat{J}_{0}(x)\right)\right)\nonumber\\
    =&-\int_x^{\infty}\left((-1)^{\mu}2J_{\mu}(t)J_0(t)\!-\!\frac{\mu}{2}\left(\hat{J}_{-\mu}(t)\!-\!(-1)^{\mu}\hat{J}_{\mu}(t)\right)J_0(t)\right)\frac{dt}{t}\!+\!\frac{(-1)^{\mu}\pi}{2}\left( Y_{\mu}(x)J_0(x)\!+\!J_{\mu}(x)Y_{0}(x)\right).
\end{align}
Finally, after reassembling all the terms, we obtain the solvable expression
\begin{align}
    &(-1)^{\mu}\left(P(x)+\int_x^{\infty} \frac{P(t)}{t}dt\right)\nonumber\\
    =&\frac{\pi\mu^2}{4}\int_x^{\infty}\frac{dt}{t}\int_t^{\infty}\frac{du}{u}\left( Y_{\mu}(u)J_{0}(u)\!+\!J_{\mu}(u)Y_{0}(u)\right)-\frac{\pi}{4}\left( Y_{\mu}(x)J_0(x)+J_{\mu}(x)Y_{0}(x)\right).
\end{align}
Setting
\begin{align}
    f_{\mu}(x)=\frac{\pi}{4}\left( Y_{\mu}(x)J_0(x)+J_{\mu}(x)Y_{0}(x)\right),\ P_{\mu}(x)=(-1)^{\mu}\mu^2\left(-f_{\mu}(x)+\int_x^{\infty}\frac{f_{\mu}(t)}{t}dt\right)+(-1)^{\mu}h_{\mu}(x),
\end{align}
we obtain
\begin{align}
    &h_{\mu}(x)+\int_x^{\infty}\frac{h(t)}{t}dt=(\mu^2-1)f_{\mu}(x)\nonumber\\
    \Rightarrow&h'_{\mu}(x)-\frac{h_{\mu}(x)}{x}=(\mu^2-1)f'_{\mu}(x)
\end{align}
It is easy to solve it via variation of parameters:
\begin{align}
    h_{\mu}(x)=x\left((1-\mu^2)\int_x^{\infty}\frac{f'_{\mu}(t)}{t}dt+C\right),
\end{align}
where $C$ is an undetermined integral constant. Then we have
\begin{align}
    P_{\mu}(x)=(-1)^{\mu}\left(\mu^2\left(-f_{\mu}(x)+\int_x^{\infty}\frac{f_{\mu}(t)}{t}dt\right)+x\left((1-\mu^2)\int_x^{\infty}\frac{f'_{\mu}(t)}{t}dt+C\right)\right)
\end{align}
Because $f_{\mu}(t)/t\to0$ when $t\to\infty$, the equation above can be further simplified by integrating by parts
\begin{align}
    P_{\mu}(x)=(-1)^{\mu}\left(-f_{\mu}(x)+\mu^2\int_x^{\infty}\frac{f_{\mu}(t)}{t}dt+(1-\mu^2)x\int_x^{\infty}\frac{f_{\mu}(t)}{t^2}dt+Cx\right).
\end{align}

\end{proof}
\section{Determination of $C$ and Asymptotic Approximation of $P_{\mu}(x)$}
To determine the integral constant in \eqref{propm}, we compare the RHS in \eqref{propm} and the original summation in \eqref{sum} as $x\to\infty$.

To obtain the asymptotic approximation of the summation in \eqref{sum}, we use the asymptotic form of $\hat{J}_{\nu}(x)$ \cite{Order} and $J_{\nu}(x)$ for large argument $x\gg\nu^2$:
\begin{align}
    \hat{J}_{\nu}(x)\sim\sqrt{\frac{\pi}{2x}}\sin\left(x-\frac{\nu\pi}{2}-\frac{\pi}{4}\right),\nonumber\\
    J_{\nu}(x)\sim\sqrt{\frac{2}{\pi x}}\cos\left(x-\frac{\nu\pi}{2}-\frac{\pi}{4}\right).\nonumber
\end{align}
Hence, the summation becomes
\begin{align}\label{trig}
    P_{\mu}(x)&\sim-\frac{1}{x}\sum_{n=1}^{N}n\left[\cos\left(2x-n\pi+\mu\frac{\pi}{2}\right)+(-1)^{\mu}\cos\left(2x-n\pi-\mu\frac{\pi}{2}\right)\right]\nonumber\\
    &\sim-N\cos\left(\frac{\mu\pi}{2}\right)\frac{\cos(2x)}{x},
\end{align}
where $x\gg N$ and $N\to\infty$. Then we have
\begin{align}\label{abs}
    \abs{P_{\mu}(x)}<\cos\left(\frac{\mu\pi}{2}\right)\frac{1}{N}\to0.
\end{align}
Now we consider the expression in \eqref{propm}. Because $f_{\mu}(x)\sim1/x$ when $x\to\infty$. Thus the first three terms are at order of $O(x^{-1})$ and $O(x^{-2})$, which means $\abs{P_{\mu}(x)}\sim a/x+Cx$ where $a$ is also a constant like $C$. Comparing this result to \eqref{abs}, we instantly get $C=0$. 

The integrals in \eqref{propm} can be evaluated as well. Following the method by using Mellin transform in \cite{chen} and finding useful Mellin transform formulae in \cite{tables}, the closed-form expression can be represented in terms of Meijer G functions:
\begin{align}
    P_{\mu}(x)=\frac{1}{4}  (-1)^{\mu +1} &\left( \pi  \left(Y_0(x) J_{\mu }(x)+J_0(x) Y_{\mu }(x)\right)+\sqrt{\pi }\mu ^2 G_{2,4}^{3,0}\left(x^2\left|
\begin{array}{c}
 \frac{1}{2},1 \\
 -\frac{\mu }{2},\frac{\mu }{2},\frac{\mu }{2},-\frac{\mu }{2} \\
\end{array}
\right.\right)\right.\nonumber\\
&\ \ +\left.\sqrt{\pi }\left(1-\mu ^2 \right)x G_{2,4}^{3,0}\left(x^2\left|
\begin{array}{c}
 -\frac{1}{2},1 \\
 -\frac{\mu +1}{2} ,\frac{\mu -1}{2},\frac{\mu -1}{2},-\frac{\mu +1}{2} \\
\end{array}
\right.\right)\right).
\end{align}

After we totally determine the expression of $P_{\mu}(x)$, we can analyze its asymptotic behavior at large $x$. Firstly, the asymptotic form of $f_{\mu}(x)$ shows
\begin{align}
    f_{\mu}(x)\sim-\frac{1}{2 x}\cos\left(2x-\frac{\mu\pi}{2}\right).
\end{align}
Then the asymptotic expression of integrals in \eqref{prop1} can be found by integrating by part, e.g.
\begin{align}
    \int_x^{\infty}\frac{f_{\mu}(t)}{t}dt&\sim-\int_x^{\infty}\cos\left(2x-\frac{\mu\pi}{2}\right)\frac{dt}{t^2}=\frac{1}{2x^2}\sin\left(2x-\frac{\mu\pi}{2}\right)-\int_x^{\infty}\sin\left(2x-\frac{\mu\pi}{2}\right)\frac{dt}{t^3}\nonumber\\
    &\sim O(x^{-2}).\nonumber
\end{align}
Both of integrals can be neglected due to $\sim O(x^{-2})$ at large $x$. Therefore, we have
\begin{align}
    P_{\mu}(x)\sim\frac{(-1)^{\mu}}{2x}\cos\left(2x-\frac{\mu\pi}{2}\right).
\end{align}
Comparing to the asymptotic form $Q(x)$ in \cite{chen}, we have $P_0(x)\sim -\frac{1}{2}Q(x)$ at large $x$. However, $P_0(x)\neq -\frac{1}{2}Q(x)$, which means that the result in \cite{chen} cannot be considered as a special case of $P_{\mu}(x)$ when $\mu=0$. It is not hard to find this argument because $P_{0}(x)$ contains the Meijer G function $G_{2,4}^{3,0}$ but $Q(x)$ contains $G^{2,0}_{1,3}$. The reason is from \eqref{lammam2} in Lemma II.2. In this paper, we only consider the cases when $\mu\neq0$, such that any time \eqref{lammam2} is applied, the RHS of this identity is automatically zero. However, if $\mu=0$ in the derivation of \eqref{propm} in Proposition II.3, $\delta_{00}=1$ can make a significant difference of the result.

\bibliography{mybib}
\bibliographystyle{unsrt}

\end{document}